\documentclass[12pt]{amsart}
\usepackage{amsmath}
\usepackage{amsthm}
\usepackage{amssymb}
\usepackage{hyperref}
\usepackage{color}
\usepackage{fouridx}
\usepackage[all]{xy}
\usepackage{mathrsfs}
\usepackage[usenames]{xcolor}
\usepackage{enumitem}

\newcommand{\gothic}{\mathfrak}

\newcommand{\Q}{{\mathbb{Q}}}
\newcommand{\m}{{\gothic{m}}}

\newcommand{\pd}{\operatorname{pd}}
\newcommand{\Spec}{\operatorname{Spec}}
\newcommand{\Supp}{\operatorname{Supp}}

\newcommand{\Ext}{\operatorname{Ext}}
\newcommand{\rk}{\operatorname{rank}}

\newcommand{\CM}{\operatorname{CM}}

\newcommand{\Cl}{\operatorname{Cl}}
\newcommand{\Hom}{\operatorname{Hom}}

\renewcommand{\hat}{\widehat}
\renewcommand{\bar}{\overline}
\renewcommand{\phi}{\varphi}
\renewcommand{\tilde}{\widetilde}

\DeclareMathOperator{\CH}{H}
\DeclareMathOperator{\Proj}{Proj}

\newcommand{\ul}[1]{\underline{#1}}

\newcommand{\upe}[1]{{}^e{#1}}
\newcommand{\sdim}{\operatorname{sdim}}
\newcommand{\dep}{\operatorname{depth}}
\newcommand{\spec}{\operatorname{Spec}}
\newcommand{\hgt}{\operatorname{ht}}
\newcommand{\chom}{\operatorname{Hom}}
\newcommand{\add}{\operatorname{add}}
\newcommand{\modcat}{\operatorname{mod}}
\newcommand{\calp}{\mathcal{P}}
\newcommand{\FF}{\mathcal{FT}}
\newcommand{\cend}{\operatorname{End}}
\newcommand{\st}{\textasteriskcentered}
\DeclareMathOperator{\bproj}{\mathbf{Proj}}
\DeclareMathOperator{\bspec}{\mathbf{Spec}}
\newcommand{\scr}[1]{\mathscr{#1}}
\newcommand{\refl}{\operatorname{Ref}}
\definecolor{orange}{rgb}{1,0.5,0}

\theoremstyle{plain}
\newtheorem{thm}{Theorem}
\newtheorem{cor}[thm]{Corollary}
\newtheorem{prop}[thm]{Proposition}
\newtheorem{lemma}[thm]{Lemma}

\newtheorem{eg}[thm]{Example}

\theoremstyle{definition}
\newtheorem{defn}[thm]{Definition}
\newtheorem{notn}[thm]{Notation}

\theoremstyle{remark}
\newtheorem{rmk}[thm]{Remark}
\newtheorem{question}[thm]{Question}
\newtheorem{discuss}[thm]{Discussion}

\begin{document}
\title[Finite $F$-type and $F$-abundant modules]
{Finite F-type and F-abundant modules}
\author{Hailong Dao}
\address{Department of Mathematics\\
University of Kansas\\
 Lawrence, KS 66045-7523 USA}
\email{hdao@ku.edu}

\author{Tony Se}
\address{Department of Mathematics\\
University of Kansas\\
 Lawrence, KS 66045-7523 USA}
\email{tonyse@ku.edu}

\dedicatory{Dedicated to Professor Kei-ichi Watanabe on the occasion of his seventieth birthday}

\date{\today}
\thanks{The first author is partially supported by NSF grant 1104017}

\bibliographystyle{amsplain}

\numberwithin{thm}{section}. 
\numberwithin{equation}{section}

\begin{abstract}
In this note we introduce and study basic properties of two types of modules over a commutative noetherian ring $R$ of positive prime characteristic. The first is the category of modules of finite $F$-type.  These objects include reflexive ideals representing torsion elements in the divisor class group of $R$. The second class is what we call $F$-abundant modules. These include, for example, the ring $R$ itself and the canonical module when $R$ has positive splitting dimension. We prove various facts about these two categories and how they are related, for example that $\Hom_R(M,N)$ is maximal Cohen-Macaulay when $M$ is of finite $F$-type and $N$ is $F$-abundant, plus some extra (but necessary) conditions. Our methods allow us to extend previous results by Patakfalvi-Schwede, Yao and Watanabe. They also afford a deeper understanding of these objects, including complete classifications in many cases of interest, such as complete intersections and invariant subrings. 
\end{abstract}

\maketitle

\section{Introduction}
\label{intro}

Let $(R,\mathfrak{m},k)$ be a
reduced  $F$-finite local
ring of dimension $d$ and prime characteristic $p>0$.
Let $\alpha(R) = \log_p [k:k^p]$. All modules are in
$\modcat(R)$, the category of finitely
generated $R$-modules.

In this paper we prove some connections between two types of objects defined using the Frobenius endomorphism $R \to R$ with $r \mapsto r^p$.

These objects generalize some well-studied concepts. Let us start with the definitions. Fix an $R$-module $M$. 
For $e \in \mathbb{Z}_{\geqslant 0}$, let $\upe{M}$ denote the
abelian group $M$ viewed as an $R$-module
via the $e$th iteration of the Frobenius map.
We let $F^e \colon \modcat(R) \to \modcat(R)$ denote the
$e$th Peskine-Szpiro functor given by
\[
  F^e(M):= M\otimes_R \upe{R}.
\]
Given $S \subseteq \modcat(R)$,
we use $\add_R(S)$ to
denote the additive subcategory
of $\modcat(R)$ generated by $S$.

\begin{defn} \label{defn:intro}
\begin{enumerate}[label=(\arabic*),align=left,leftmargin=*,nosep]
\item Let $M$ be an $R$-module
such that $\Supp(M) =\Spec R$
and is locally free in
codimension 1.
We let $M(e)
= F^e_R(M)^{**}$, the reflexive hull of $F^e_R(M)$, viewed
as an $R$-module by
identifying $\upe{R}$ with $R$.
We say that $M$ is of
finite $F$-type if
$\{M(e)\}_{e \geqslant 0}
\subseteq \add_R(X)$ for some
$R$-module $X$ (see Lemma~\ref{lem:addx}).
We let $\FF(R)$ denote the
category of $R$-modules
of finite $F$-type.

\item Let $N,L$ be $R$-modules.
Let $b_e$ be maximum such
that $\upe{N} = L^{\oplus b_e}
\oplus N_e$ for some $N_e$. We say
that $(N,L)$ is an
abundant pair
if $\liminf_{e \to \infty}
p^{e\alpha(R)}/b_e = 0$.

\item Let $L$ be an $R$-module.
We say that $L$ is an $F$-abundant
module if $(N,L)$ is an
abundant pair for some $N$.
\end{enumerate}
\end{defn}

Examples of modules of finite $F$-type include torsion elements of the divisor class group of a normal domain $R$ (without any assumption about the order of the element), finite integral extensions that are \'etale in codimension one (see section~\ref{sec:ft}),  or $F$-periodic vector bundles on the punctured spectrum of $R$ (and of the corresponding projective variety $X$ when $R$ is a local cone of some embedding of $X$; see section~\ref{sec:vecbun}). For $F$-abundant modules, $R$ has positive splitting dimension if and only if $(R,R)$ is an abundant pair (see section~\ref{sec:abun}). A good source of examples in both cases are the rings of invariants of a finite group, see Example~\ref{HNex}.

Our main technical results (collected in section \ref{mainSec})  say roughly that under various extra conditions, if $M$ is of finite $F$-type and $N$ is $F$-abundant, then $\Hom_R(M,N)$ is maximal Cohen-Macaulay. We shall give plenty of examples to show that the technical conditions are necessary. Our approach yields a strengthening of some well-known results as well as many new ones: 

\begin{enumerate}[label=(\arabic*),align=left,leftmargin=*]
\item An extension of results by Patakfalvi-Schwede (\cite[Theorem 3.1]{PS})
and Watanabe (\cite[Corollary 2.9]{Wa}) on depth of divisor classes associated to
$F$-regular singularities. See Theorem \ref{PS}. 

\item Under certain conditions, a strong generalization
of Yao's result (\cite[Lemma 2.2]{Yao}) on Cohen-Macaulayness
of $F$-contributors. See Theorems~\ref{thm:mainweak} and
\ref{thm:mainstrong}.

\item
A complete classification of the categories of finite $F$-type $R$-modules
and abundant $F$-modules in many cases of interest, such as complete intersections and invariant subrings (Corollary \ref{cor:inva}, Theorem  \ref{thm:cift} and \ref{sec:abun}).
\end{enumerate}

\section{Notation and preliminary results}

\begin{defn} (\cite[Definition 2.4]{AE} and
\cite[Definition 4.5]{BST})
Let $a_e = a_e(R)$ be maximum such that $\upe{R} =
R^{\oplus a_e} \oplus R_e$ for some $R_e$.
The largest integer $k$ such that
\[
  \lim_{e \to \infty}
  \frac{a_e}{p^{e(k + \alpha(R))}}
  > 0
\]
is called the $F$-splitting
dimension of $R$, and
is denoted by $\sdim R$.
\end{defn}

\begin{notn}
\begin{enumerate}[label=(\alph*),align=left,leftmargin=*]
\item We use $(S_k)$ to denote
Serre's criteria for $k
\geqslant 0$.
\item We use $\CH^i_{\m}(M)$ to
denote the $i$th local cohomology
module of an $R$-module $M$
supported on $\m$.
\item For an $R$-module $M$,
we let $M^*$ denote
$\chom_R(M,R)$.
\item Suppose that $R$ has a
canonical module $\omega_R$.
Then for an $R$-module $M$,
we let $M^{\vee}$ denote
$\chom_R(M,\omega_R)$.
\item For $S \subseteq R$, we
let $\upe{S}$ to denote
$S$ viewed as a subset
of $\upe{R}$. Then for
$P \in \spec R$,
$\fourIdx{e}{}{}{\,{}^eP}{R}$ is
an $R_P$-module, and
$\fourIdx{e}{}{}{\,{}^eP}{R} =
\upe{(R_P)}$.
\item For an ideal $I \subseteq R$,
an integer $e \geqslant 0$ and $q=p^e$, we
let $I^{[q]} = \{x^q \mid x \in I\}$,
the $q$th Frobenius power of $I$.
\item We use $\CM(R)$ to denote
the subcategory of $\modcat(R)$
consisting of all
C-M $R$-modules and $\refl(R)$
the subcategory of all
reflexive modules.
\end{enumerate}
\end{notn}

\begin{rmk}
We would like to remind the reader
of another characterization
of $\sdim R$.
We consider the splitting prime
$\calp(R)$ of $R$, as defined in
\cite[Definition~3.2]{AE}.
By \cite[Theorem~3.3 and Corollary~3.4]{AE},
$\calp(R)$ is a prime ideal if $\sdim(R) \neq
-\infty$ or the unit ideal otherwise.
Corollary~4.3 of \cite{BST} shows that
$\sdim R = \dim (R/\calp(R))$
when $\sdim(R) \neq -\infty$.
\end{rmk}

\begin{prop} \label{prop:cmpunc}
Assume the following for $R$:
\begin{enumerate}[label=(\arabic*),align=left,leftmargin=*]
  \item $R$ is equidimensional;
  \item $R_P$ is C-M for all
  $P \in \spec R \setminus
  \{\mathfrak{m}\}$; and
  \item $\sdim R > 0$.
\end{enumerate}
Then $R$ is C-M.
\end{prop}

\begin{proof}
Since $R$ is $F$-finite,
it is a homomorphic image
of a Gorenstein ring by
\cite[Remark before Lemma~A.2]
{HMS}. By (1) and (2),
for $0 \leqslant i < d$
we have that $\CH_{\mathfrak{m}}^i
(R)$ has finite length
by \cite[Proposition~21.24]
{24hrs}; see also
\cite[Theorem~9.5.2 (Grothendieck's
Finiteness Theorem)]{lc}. So
{\allowdisplaybreaks
\begin{align*}
   p^{e\alpha(R)} \lambda_R (\CH_{\m}^i(R))
  &= \lambda_R (\upe{\CH_{\m}^i(R)})\\
  &= \lambda_R (\CH_{\m}^i (\upe{R}))\\
  &= \lambda_R (\CH_{\m}^i (R^{\oplus a_e} \oplus M_e))\\
  &= a^e \lambda_R (\CH_{\m}^i (R))
  + \lambda_R (\CH_{\m}^i (M_e))\\
  \frac{1}{p^e} \lambda_R (\CH_{\m}^i(R))
  &= \frac{a_e}{p^{e(1+\alpha(R))}}
  \lambda_R (\CH_{\m}^i (R))
  + \frac{1}{p^{e(1+\alpha(R))}}
  \lambda_R (\CH_{\m}^i (M_e))\\
  0 = \lim_{e \to \infty}
  \frac{1}{p^e} \lambda_R (\CH_{\m}^i(R))
  &\geqslant
  \lim_{e \to \infty}
  \frac{a_e}{p^{e(1+\alpha(R))}}
  \lambda_R (\CH_{\m}^i (R))
\end{align*}}%
Since $\sdim R \geqslant 1$,
we have $\lambda_R (\CH_{\m}^i (R)) = 0$
and hence $\CH_{\m}^i(R) = 0$ for $0 \leqslant i
< d$. Then by \cite[Theorem~10.36]{24hrs}, $R$ is C-M.
\end{proof}

\begin{cor} \label{cor:sdim0}
If $R$ is $F$-split, is
C-M on $\spec R \setminus
\{\mathfrak{m}\}$ (e.g.\ $R$
is an isolated singularity)
but not C-M, then $\sdim R
= 0$. \qed
\end{cor}

\begin{eg}
Let $X$ be an ordinary Abelian
variety of dimension at least two
and $R$ be the coordinate
ring of an embedding of $X$
with respect to some polarization.
It is well-known, as in \cite{MS,Smi},
that $R$ is $F$-split but not C-M.
So by Corollary~\ref{cor:sdim0},
$\sdim R = 0$.
\end{eg}


\begin{lemma} \label{lem:Yao}
Let $M,N$ be $R$-modules such that $^eM=N^{b_e}\oplus M_e$ and  $\liminf_{e \to \infty} \frac{b_e}{p^{e(k+\alpha(R))}} >0$. Then $\dep N \geqslant k$. 
In particular, if $k = \dim(M)$, then $N$ is C-M.
\end{lemma}

\begin{proof}
We use the same proof as in \cite[Lemma~2.2]{Yao}.
For $0 \leqslant i < k$,
we have
\[
  0 =
  \lim_{e \to \infty}
  \frac{\lambda_R (
  \CH^i_R (\ul{x}^{p^e},M))}
  {p^{ek}}
  = \lim_{e \to \infty}
  \frac{\lambda_R (
  \CH^i_R (\ul{x},\upe{M}))}
  {p^{e(k+\alpha(R))}}
  \geqslant
  \liminf_{e \to \infty}
  \frac{b_e}
  {p^{e(k+\alpha(R))}}
  \lambda_R (
  \CH^i_R (\ul{x},N))
\]
as in Proposition~\ref{prop:cmpunc}.
So $\CH^i_R (\ul{x},N)=0$
for all $0 \leqslant i < k$,
and hence $\dep N \geqslant k$.
\end{proof}

\begin{rmk}
It is already known that $\sdim R = d
\Rightarrow R$ is strongly
$F$-regular $\Rightarrow$
$R$ is C-M.
\end{rmk}

\begin{rmk}
Let $\upe{M} = R^{b_e} \oplus M_e$ with $b_e$
the largest possible. The largest integer $k$
such that
\[
  \liminf_{e \to \infty} \frac{b_e}{p^{e(k+\alpha(R))}}
  >0
\]
was defined in \cite[Definition 5.4]{AE} to be
$\sdim(M)$, the $s$-dimension of $M$.
\end{rmk}

\begin{defn}
We will let $(\sdim_n)$ to
denote the statement:
$\sdim R_P > 0$ for all
$P \in \spec R$ such that
$\hgt P \geqslant n$.
\end{defn}

\begin{lemma}
If $R$ satisfies $(\sdim_n)$,
then $\hgt \calp(R) < n$.
\end{lemma}

\begin{proof}
Suppose that $R$ satsifies
$(\sdim_n)$. By \cite[Proposition~3.6]
{AE}, we have
$\calp(R_{\calp(R)}) =
\calp(R) R_{\calp(R)}$.
Since $R$ is $F$-finite
and reduced, so is
$R_{\calp(R)}$.
By \cite[Corollary~3.4]{AE},
since $\calp(R_{\calp(R)})$
is the maximal ideal of
$R_{\calp(R)}$, we have
$\sdim R_{\calp(R)} = 0$,
so $\hgt \calp(R) < n$.
\end{proof}

\section{Main technical results}\label{mainSec}

\begin{lemma} \label{lem:isocodim1}
Let $f \colon M \to N$
be a homomorphism of $R$-modules,
where $M$ is $(S_2)$ and
$N$ is $(S_1)$. Suppose
that $f$ is an isomorphism
in codimension 1. Then
$f$ is an isomorphism.
\end{lemma}

\begin{rmk}
In many of the results below,
we can replace the assumption of
a module being locally free
in codimension 1 by $R$
being quasinormal.
\end{rmk}

\begin{cor} \label{cor:dbldual}
Let $L,N$
be $R$-modules such
that $L$ is
free in codimension 1
and $N$ is $(S_2)$.
Then
\[
  \chom_R(L,N) 
  \cong
  \chom_R(L^{**},N)
\]
\end{cor}

\begin{proof}
Let $f \colon L \to L^{**}$
be the canonical map and
$\bar{f} \colon \chom_R(L^{**},
N) \to \chom_R(L,N)$ the
induced map. Then $\bar{f}$
is an isomorphism in
codimension 1. By
Lemma~\ref{lem:isocodim1},
it suffices to show that
$\chom_R(M,N)$ is $(S_2)$
for every $R$-module $M$.
Given $M$, let $F_1
\to F_0 \to M \to 0$ be a
finite presentation of $M$.
Applying $\chom_R(-,N)$
gives
\[
  0 \to \chom_R(M,N)
  \to \chom_R(F_0,N)
  \to \chom_R(F_1,N)
\]
Since $N$ is $(S_2)$, so
are $\chom_R(F_i,N)$ for
$i = 1,2$ and
$\chom_R(M,N)$, as required.
\end{proof}

\begin{cor} \label{cor:dblcech}
Suppose that $R$ is C-M with
a canonical module $\omega_R$.
Suppose that $M$ is $(S_2)$
and $M^{\vee}$ is MCM. Then
$M$ is MCM.
\end{cor}

\begin{proof}
Consider the natural map
$f \colon M \to M^{\vee \vee}$.
In codimension 1,
$M$ is MCM since it is $(S_2)$,
so $f$ is an isomorphism.
Since $M^{\vee}$ is MCM,
so is $M^{\vee \vee}$.
Hence $f$ is an isomorphism
by Lemma~\ref{lem:isocodim1}.
\end{proof}

\begin{rmk} \label{rmk:localization}
Let $M$ be an $R$-module,
$f \colon R \to S$ be a ring
homomorphism, $U$ be a
multiplicative subset of $S$
and $T = f^{-1} (U)$.
Then we have
$(M \otimes_R S)_U
= M \otimes_R S \otimes_S S_U
= M \otimes_R S_U
= M \otimes_R R_T
\otimes_{R_T} S_U
= M_T \otimes_{R_T} S_U$.
\end{rmk}

\begin{lemma} \label{lem:dbldualringext}
Let $R \xrightarrow{f_1} R_1
\xrightarrow{f_2} R_2$
be a sequence of ring
homomorphisms where
$R_2$ is $(S_2)$.
Let $f =f_2 \circ f_1$.
Let $M$ be an $R$-module.
Suppose that $M_P$ is free
for every $P \in \spec R$
of height 1 and
for every $P = f^{-1}(Q)$
such that $Q \in \spec R_2$
of height 1.
Let $M_i = (M \otimes_R
R_i)^{**}$ (over $R_i$)
for $i=1,2$.
Then $(M_1 \otimes_{R_1}
R_2)^{**} \cong M_2$.
\end{lemma}

\begin{proof}
Let $N = M \otimes_R R_1$.
The natural map $N \to M_1
= N^{**}$ gives rise to
a map $g \colon
M \otimes_R R_2 =
M \otimes_{R} R_1
\otimes_{R_1} R_2
= N \otimes_{R_1} R_2
\to N^{**} \otimes_{R_1}
R_2 = M_1 \otimes_{R_1}
R_2$. Let $Q \in \spec R_2$,
$P_1 = (f_2)^{-1}(Q)$
and $P = f^{-1}(Q)$.
Then $(M \otimes_R R_2)_Q
= M_P \otimes_{R_P}
(R_2)_Q$ and
$(M_1 \otimes_{R_1} R_2)_Q
= (M \otimes_R R_1)^{**}_{P_1}
\otimes_{(R_1)_{P_1}} (R_2)_Q
= (M_P \otimes_{R_P} (R_1)_{P_1})^{**}
\otimes_{(R_1)_{P_1}} (R_2)_Q$,
so $g$ is an isomorphism
in codimension 1.
Applying ${}^{**}$
gives the map $M_2 = (M
\otimes_R R_2)^{**} \to
(M_1 \otimes_{R_1}
R_2)^{**}$, which is an
isomorphism by
Lemma~\ref{lem:isocodim1}
since $R_2$ is $(S_2)$.
\end{proof}

\begin{cor} \label{cor:meplusf}
(``Index shifting'')
Let $R$ be $(S_2)$ and $M$ be as in
Definition~\ref{defn:intro}.
Let $e,f$ be nonnegative
integers. Then $[M(e)](f)
\cong M(e+f)$.
\end{cor}

\begin{proof}
Let $R_1 = \upe{R}$ and $R_2
= \fourIdx{e+f}{}{}{}{R}$ as in
Lemma~\ref{lem:dbldualringext}.
Then we have $M(e) = M_1$,
and so $[M(e)](f)
= (M_1 \otimes_{R_1}
R_2)^{**} \cong M_2
= M(e+f)$.
\end{proof}

\begin{thm} \label{thm:mainweak}
Suppose that $R$ is
$(S_2)$ and
equidimensional.
Let $M,N$ be $R$-modules
such that $M \in \FF(R)$
and $N$ is $(S_2)$.
Assume that $(N,L)$
is an abundant pair
and that $N_P \in \add L_P$
for all $P \in
\spec R$ such that
$3 \leqslant \hgt(P)<d$.
Assume further
that for every $P \in
\spec R$ such that
$3 \leqslant \hgt(P)<d$
and $e \geqslant 0$,
$(\chom_R(M(e),L))_P$
is MCM.
Then $\chom_R(M(e),L)$
is MCM for all
$e \geqslant 0$.
\end{thm}

\begin{proof}
If $d \leqslant 2$, then
$\chom_R(M(e),L)$
is MCM since $L$
is $(S_2)$. So we may
assume that $d \geqslant
3$. By assumption,
for $P \in \spec R
\setminus \{\mathfrak{m}\}$,
we have $N_P \in \add L_P$
for $\hgt(P) \geqslant 3$,
so $(\chom_R(M(e),N))_P$
is MCM for $e \geqslant 0$.
So for $0 \leqslant i < d$ and
$e \geqslant 0$,
$\CH^i_{\mathfrak{m}}
(\chom_R(M(e),N))$ has
finite length
--- see Proposition~\ref{prop:cmpunc}.
By Corollary~\ref{cor:dbldual},
we have
{\allowdisplaybreaks
\begin{align*}
  \chom_{{}^e R}(\fourIdx{}{{}^e R}{}{}{[F^e_R(M)^{**}]},
  \fourIdx{}{{}^e R}{}{}{N})
  &= \chom_{{}^e R}
  (\fourIdx{}{{}^e R}{}{}{[(M \otimes_R \upe{R})^{**}]},
  \fourIdx{}{{}^e R}{}{}{N})\\
  &= \chom_{{}^e R}
  (\fourIdx{}{{}^e R}{}{}{[M \otimes_R \upe{R}]},
  \fourIdx{}{{}^e R}{}{}{N})\\
  &= \chom_R (M,
  \chom_{{}^e R}(\upe{R},
  \fourIdx{}{{}^e R}{}{}{N}))\\
  &= \chom_R (M,
  \fourIdx{}{{}^e R}{}{}{N})\\
  &= \chom_R (M,
  \upe{N})\\
  &= \chom_R (M,
  L^{\oplus b_e} \oplus
  N_e)\\
  &= \chom_R(M,
  L)^{\oplus b_e}
  \oplus \chom_R(M,N_e)
\end{align*}}%

Apply $\CH_{\m}^i$ for $0 \leqslant i < d$
to get
{\allowdisplaybreaks
\begin{align*}
  b_e \lambda_R(\CH_{\m}^i(\chom_R(M,L)))
  &\leqslant \lambda_R(\CH_{\m}^i(
  \chom_{{}^e R}(\fourIdx{}{{}^e R}{}{}{[F^e_R(M)^{**}]},
  \fourIdx{}{{}^e R}{}{}{N})
  ))\\
  &=
  \lambda_R(
  \CH_{\mathfrak{m}({}^e R)}^i(
  \chom_{{}^e R}(\fourIdx{}{{}^e R}{}{}{[F^e_R(M)^{**}]},
  \fourIdx{}{{}^e R}{}{}{N})
  ))
  \text{\quad (base change)}\\
  &=
  \lambda_R(
  \CH_{{}^e \m}^i(
  \chom_{{}^e R}(\fourIdx{}{{}^e R}{}{}{[F^e_R(M)^{**}]},
  \fourIdx{}{{}^e R}{}{}{N})
  ))\\
  &=
  \lambda_R(
  \fourIdx{e}{}{i}{\m}
  {\CH}(
  \chom_R(M(e),
  \fourIdx{}{R}{}{}{N})
  ))\\
  &= p^{e\alpha(R)}
  \lambda_R(
  \fourIdx{}{}{i}{\m}{\CH}(
  \chom_R(M(e),
  \fourIdx{}{R}{}{}{N})
  ))\\
  \lambda_R(
  \CH_{\m}^i(\chom_R(M,L)))
  &\leqslant
  \frac{p^{e\alpha(R)}}{b_e}
  \max\{
  \lambda_R(
  \fourIdx{}{}{i}{\mathfrak{m}}
  {\CH}(
  \chom_R(M(e),
  \fourIdx{}{R}{}{}{N})
  ))\}
\end{align*}}%

Taking $\liminf$ shows
that $\CH_{\m}^i(\chom_R(M,L))=0$ for
$0 \leqslant i < d$,
so $\chom_R(M,L)$ is MCM.
By Corollary~\ref{cor:meplusf},
we have $[M(e)](f) \cong
M(e+f)$. So we may
replace $M$ by $M(e)$
to conclude that
$\chom_R(M(e),L)$ is MCM.
\end{proof}

\begin{rmk}
By Lemma~\ref{lem:addx}, we only need to check that
$(\chom_R(M_i,L))_P$ is MCM for $3 \leqslant \hgt(P)
< d$ for the finitely many indecomposable modules
$M_i$ that appear among $\{M(e)\}_{e \geqslant 0}$.
\end{rmk}

\begin{thm} \label{thm:mainstrong}
Let $R$ be as in
Theorem~\ref{thm:mainweak}.
Let $M,N$ be $R$-modules
such that $M \in \FF(R)$
and $N$ is $(S_2)$.
Assume that for every
$P \in \spec R$ such
that $\hgt (P) \geqslant 3$,
$(N_P,L_P)$ is an abundant
pair. Assume further
that for every $P \in
\spec R$ such that
$3 \leqslant \hgt (P) <d$,
we have $N_P \in \add L_P$.
Then $\chom_R(M(e),L)$
is MCM for all $e \geqslant 0$.
\end{thm}

\begin{proof}
First, for every
$P \in \spec R$, we
have the following.
{\allowdisplaybreaks
\begin{align*}
  \fourIdx{}{{}^e (R_P)}{}{}{[F^e_{R_P}(M_P)^{**}]}
  &= \chom_{{}^e (R_P)}
  \big(\chom_{{}^e (R_P)}
  \big(M_P \otimes_{{}^e (R_P)}
  \upe{(R_P)}, \upe{(R_P)}
  \big), \upe{(R_P)} \big) \\
  &= \chom_{{}^e (R_P)}
  \big(\chom_{{}^e (R_P)}
  \big(M \otimes_R R_P
  \otimes_{{}^e (R_P)}
  \upe{(R_P)}, \upe{(R_P)}
  \big), \upe{(R_P)} \big) \\
  &= \chom_{{}^e (R_P)}
  \big(\chom_{{}^e (R_P)}
  \big(M \otimes_R
  \upe{(R_P)}, \upe{(R_P)} \big)
  , \upe{(R_P)} \big) \\
  &= \chom_{{}^e (R_P)}
  \big(\chom_{{}^e R_{\,{}^e P}}
  \big(M \otimes_R
  \upe{R} \otimes_{{}^e R_{\,{}^e P}}
  \upe{R_{\,{}^e P}},
  \upe{R_{\,{}^e P}} \big),
  \upe{(R_P)} \big) \\
  &= \chom_{{}^e (R_P)}
  \big(\chom_{{}^e R}
  \big(M \otimes_R
  \upe{R},
  \upe{R} \big)_{\,{}^e P},
  \upe{(R_P)} \big) \\
  &= \Big( \chom_{{}^e R}
  \big(\chom_{{}^e R}
  \big(M \otimes_R
  \upe{R},
  \upe{R} \big),
  \upe{R} \big) \Big)_{\,{}^e P} \\
  &= (F^e_R(M)^{**})_P
\end{align*}}%

So we can prove by
induction on $d$ that
$\chom_R (M(e),L)$ is MCM
for all $e \geqslant 0$.
We may assume that $d
\geqslant 3$.
Let $P \in \spec{R}$ such
that $3 \leqslant \hgt(P)
< d$. By induction,
$(\chom_R(M(e),L))_P$
is MCM. So by
Theorem~\ref{thm:mainweak},
$\chom_R(M(e),L)$
is MCM.
\end{proof}

\begin{cor} \label{cor:thmweak}
Suppose that $R$ is C-M and $M \in \FF(R)$ is $(S_2)$.
Assume that:
\begin{enumerate}[label=(\alph*),align=left,leftmargin=*,nosep]
  \item either $\sdim R > 0$ and
  $M(e)_P$ is MCM for every 
  $P \in \spec R$ such that
  $3 \leqslant \hgt(P) < d$ and
  $e \geqslant 0$; or \label{Me_P}
  \item $R$ is $(\sdim_3)$. \label{sdim}
\end{enumerate}
Then $M$ is MCM.
\end{cor}

\begin{proof}
Since $R$ is a homomorphic
image of a Gorenstein ring,
it has a canonical module
$\omega_R$ by \cite[12.1.3(iii)]{lc}.
Since $R$ is C-M, it is
equidimensional, and $\omega_R$
is reflexive and hence $(S_2)$.
We will show that
$\chom_R (M,\omega_R)
= M^{\vee}$ is MCM, so that
by Corollary~\ref{cor:dblcech},
$M$ is MCM.
We may assume that $d > 2$.
Since $\sdim R > 0$, $(R,R)$
is an abundant pair by
Example~\ref{eg:abdtpair}.
Then $(\omega_R,
\omega_R)$ is also an abundant
pair, since $\omega_{{}^e R}
\cong \chom_R(\upe{R},
\omega_R) = \chom_R(R^{\oplus a_e}
\oplus R_e, \omega_R)
= \omega_R^{\oplus a_e}
\oplus \chom_R(R_e,
\omega_R)$. If \ref{Me_P} holds, then
$(\chom_R(M(e),\omega_R))_P$
is MCM for every 
$P \in \spec R$ such that
$3 \leqslant \hgt(P) < d$ and
$e \geqslant 0$.
By Theorem~\ref{thm:mainweak},
$\chom_R (M,\omega_R)
= M^{\vee}$ is MCM.
If \ref{sdim} holds, then
$(\omega_{R_P},\omega_{R_P})$
is an abundant pair for
$P \in \spec R$ such that
$\hgt(P) \geqslant 3$, so
$\chom_R (M,\omega_R)$ is MCM
by Theorem~\ref{thm:mainstrong}.
\end{proof}

\begin{cor} \label{cor:idealmcm}
Suppose that $R$ is strongly $F$-regular and
$M \in \FF(R)$ is $(S_2)$.
Then $M$ is MCM. In particular, if
$I$ is a reflexive ideal
such that $[I]$ is torsion
in $\Cl(R)$, then $I$ is MCM.
\end{cor}

\begin{proof}
Since $R$ is strongly
$F$-regular, it is $(\sdim_3)$.
Corollary~\ref{cor:thmweak}
shows that $M$ is MCM.
By Example~\ref{eg:modft}, we have
$I \in \FF(R)$.
\end{proof}

\begin{rmk}
Corollary~\ref{cor:idealmcm}
generalizes \cite[Corollary~3.3]{PS}.
\end{rmk}

\section{The category of finite \texorpdfstring{$F$}{F}-type
modules} \label{sec:ft}
In this section we study the category of finite $F$-type in more detail. We completely classify this category when $R$ is a quotient singularity with a finite group whose order is coprime to $p$ (Corollary \ref{quotient}), or when $R$ is a complete intersection that is regular in codimension $2$ (Theorem \ref{thm:cift}).

\begin{eg} \label{eg:modft}
The following are $R$-modules
$M$ of finite $F$-type.\\
(a) (\cite[paragraph~2.3]{Wa})
$R$ is a normal
domain and $M = I$, where
$I$ is a divisorial ideal.
Then $M(e) \cong I^{(e)}$,
so $M \in \FF(R)$ iff
$[I]$ is torsion
in $\Cl (R)$.\\
(a') $M$ is a free $R$-module.\\
(b) (\cite[Theorem~2.7]{Wa})
Let $R
\to S$ be a
finite homomorphism
of normal domains which is \'{e}tale in
codimension $1$ and $M=S$.
From the natural map
$S \otimes_R \upe{R} \to
\upe{S}$ we get
$(S \otimes_R \upe{R})^{**}
\cong S^{**}$, i.e.\ $M(e)
\cong S$.\\
(c) See also Lemma~\ref{lem:fperiodic}.
\end{eg}

\begin{lemma}
$\FF(R)$ is closed under direct
sums and direct summands.
\end{lemma}

\begin{proof}
Obvious.
\end{proof}

\begin{lemma} \label{lem:addx}
Let $S \subseteq \modcat R$.
Then $\add_R(S)$ has finitely many
indecomposable objects iff
$S \subseteq \add_R(X)$
for some $R$-module $X$.
Hence for an $R$-module $M$,
$M \in \FF(R)$ if and only if  only finitely
many indecomposable direct
summands appear among $\{M(e)
\}_{e \geqslant 0}$.
\end{lemma}

\begin{proof}
For the ``only if'' part, take
$X$ to be the direct sum
of the indecomposable objects
of $\add_R(S)$. For the
``if'' part, we consider
the endomorphism ring $E =
\cend_R(X)$. Consider the
category $\Proj E \subseteq
\modcat (E)$ of left projective
modules over $E$. Then $F \colon
\add_R(X) \to \Proj E$ given
by $F(L) = \chom_R(L,X)$ for
$L \in \add_R(X)$ is an
equivalence of categories.
Since $E$ is finitely
generated over the local ring $R$,
it is semilocal.
By \cite[Theorem~9]{FS},
$\Proj E$ has only finitely
many isomorphism classes of
indecomposable objects,
and hence so does $\add_R(X)$.
\end{proof}

\begin{cor} \label{cor:periodicity}
Let $R,M$ be as in
Corollary~\ref{cor:meplusf}. Then
$M \in \FF(R)$ iff there are
$e \geqslant 0$ and $f>0$
such that $M(e) \cong M(e+f)$.
\end{cor}

\begin{proof}
``If'': We note that
$M(e+g) \cong [M(e)](g)
\cong [M(e+f)](g) \cong
M(e+f+g)$ by
Corollary~\ref{cor:meplusf},
so there are only finitely
many isomorphism classes
of $\{M(e)\}_{e \geqslant 0}$.\\
``Only if'': Let $X$ be
an $R$-module such that
$\{M(e)\}_{e \geqslant 0} \subseteq
\add_R(X)$. Let $S$ be the set of
indecomposable direct summands
of $\{M(e)\}_{e \geqslant 0}$.
Then $S \subseteq \add_R(X)$,
and $S$ is finite by
Lemma~\ref{lem:addx}. By index
shifting, it suffices
to prove that $N(e) \cong N(e+f)$ for
some $e \geqslant 0$ and
$f > 0$ for all $N \in S$.
By assumption, $R$ is local, so we can prove
the claim by induction on the minimum number
of generators of $N$. Suppose first that
$N(e)$ is indecomposable for all $e$.
Then the claim holds since $S$ is finite.
So suppose that $N(e_0)$ is
not indecomposable for some
$e_0$ and has indecomposable
direct summands $N_i \in S$.
By induction, the claim holds
for the $N_i$, and hence for
$N$ by index shifting.
\end{proof}

\begin{prop} \label{prop:closureft}
Let $R$ be $(S_2)$.
Let $\FF=\FF(R)$. Then:\\
(a) if $M \in \FF$, then
$M^{**} \in \FF$.\\
(b) if $M,N \in \FF$,
then $M \otimes N \in \FF$.
\end{prop}

\begin{proof}
(a) By Corollary~\ref{cor:meplusf},
we have $[M(0)](e) \cong
M(e)$, that is, $(M^{**})(e) \cong
M(e)$.\\
(b) For $P \in \spec R$,
we have $(M \otimes_R N)_P
= M_P \otimes_{R_P} N_P$,
so $M \otimes N$ is locally
free in codimension 1. Next,
we have
\begin{align*}
  M \otimes_R N \otimes_R \upe{R}
  &= M \otimes_R N \otimes_R
  \upe{R} \otimes_{{}^e R}
  \upe{R}\\
  &= M \otimes_R \upe{R}
  \otimes_{{}^e R} \upe{R}
  \otimes_R N\\
  &= (M \otimes_R \upe{R})
  \otimes_{{}^e R} (N \otimes_R
  \upe{R})\\
  (M \otimes_R N \otimes_R
  \upe{R})^{**}
  &= [(M \otimes_R \upe{R})
  \otimes_{{}^e R} (N \otimes_R
  \upe{R})]^{**}
\end{align*}
The natural maps $M \otimes_R
\upe{R} \to (M \otimes_R
\upe{R})^{**}$ and
$N \otimes_R
\upe{R} \to (N \otimes_R
\upe{R})^{**}$ give rise
to the map $f \colon
[(M \otimes_R \upe{R})
\otimes_{{}^e R} (N \otimes_R
\upe{R})]^{**} \to
[(M \otimes_R \upe{R})^{**}
\otimes_{{}^e R} (N \otimes_R
\upe{R})^{**}]^{**}$. Since
$R$ is $(S_2)$,
Lemma~\ref{lem:isocodim1} shows
that $f$ is an isomorphism.
So we have an isomorphism
\[
  (M \otimes_R N)(e)
  \cong
  (M(e) \otimes_R N(e)
  )^{**}
\]
By Lemma~\ref{lem:addx}, only
finitely many indecomposable
direct summands
$\{K_i\}_{i=1}^m$ appear in
$\{M(e)\}_{e \geqslant 0}$
and $\{L_j\}_{j=1}^n$ in
$\{N(e)\}_{e \geqslant 0}$.
Let $X = \sum_{i,j}
(K_i \otimes_R L_j)^{**}$.
Then $\{(M \otimes_R N)(e)
\}_{e \geqslant 0} \subseteq
\add_R(X)$.
\end{proof}

\begin{question}
If $R$ is Gorenstein,
is $M^* \in \FF(R)$?
\end{question}

\begin{lemma} \label{lem:basechangeft}
Let $f \colon R \to S$ be
a ring homomorphism.
Suppose that $S$ is $(S_2)$. Let
$M$ be an $R$-module.
Suppose that:\\
(a) $f$ is flat; or\\
(b) $M_P$ is free
for every $P = f^{-1}(Q)$
such that $Q \in \spec S$
and $\hgt (Q) = 1$.\\
If $M \in \FF(R)$, then
$M \otimes_R S \in \FF(S)$.
\end{lemma}

\begin{proof}
Let $N = M \otimes_R S$.
First, we consider (b)
and suppose that $Q
\in \spec S$, $\hgt(Q)=1$
and $M \in \FF(R)$.
Let $P = f^{-1}(Q)$.
Then $N_Q
= M_P \otimes_{R_P} S_Q$,
which is free over $S_Q$
by assumption. Next, we
consider the modules
$F^e_S(N)^{**}
= (M \otimes_R S \otimes_S
\upe{S})^{**}
= (M \otimes_R \upe{S})^{**}
= (M \otimes_R \upe{R}
\otimes_{{}^e R} \upe{S})^{**}$.

\[
  \xymatrix{
  R \ar[r]^f \ar[d]_{\phi}
  & S \ar[d]^{\phi}\\
  \upe{R} \ar[r]^f
  & \upe{S}
  }
\]
Consider the maps
$R \xrightarrow{\phi} \upe{R}
\xrightarrow{f} \upe{S}$,
where $\phi$ is the
Frobenius map,
and let $R_1 = \upe{R}$,
$R_2 = \upe{S}$. Then
Lemma~\ref{lem:dbldualringext} shows
that $F^e_S(N)^{**}
= (M \otimes_R \upe{S})^{**}
\cong (F^e_R(M)^{**}
\otimes_{{}^e R} \upe{S})^{**}$,
so $N(e) \cong
(M(e) \otimes_R S)^{**}
\subseteq \add_S(
(X \otimes_R S)^{**})$.\\
Now suppose that $f$ is flat.
Let $Q$, $P$ and $M$ be as
above. Then we have
$\hgt(Q) = \hgt(P) + \dim
(S_Q/PS_Q)$, so $\hgt(P)
\leqslant 1$ and $M_P$
is free, giving (b).
\end{proof}

\begin{lemma} \label{lem:mstar}
Suppose that $R$ is regular.
Let $M$ be an $R$-module.
Then $M^*$ is reflexive.
\end{lemma}

\begin{proof}
Consider the canonical map
$f \colon M^* \to M^{***}$.
In codimension 1, $R$ is
a principal ideal domain,
so $M^*$ is free, and $f$
is an isomorphism. Since
$R$ is $(S_2)$, so are
$M^*$ and $M^{***}$. By
Lemma~\ref{lem:isocodim1},
$f$ is an isomorphism.
\end{proof}

\begin{lemma} \label{lem:basechangerfl}
Let $R \to S$ be a flat
ring extension. Suppose
that $M$ is a reflexive
$R$-module. Then $M \otimes_R S$
is a reflexive $S$-module.
\end{lemma}

\begin{proof}
Since $R \to S$ is flat,
we have $M \otimes_R S
= M^{**} \otimes_R S
= (M \otimes_R S)^{**}$.
\end{proof}

\begin{lemma} \label{lem:regft}
Suppose that $R$ is regular.
Consider the following
statements:\\
(a) $M \in \FF(R)$\\
(b) $M^*$ is free.\\
(c) $M^{**}$ is free.\\
Then (a) $\Rightarrow$
(b) $\Leftrightarrow$ (c).
If $M$ is free in codimension 1,
then (a) $\Leftrightarrow$
(b) $\Leftrightarrow$ (c).
\end{lemma}

\begin{proof}
(b) $\Rightarrow$ (c): Obvious.\\
(c) $\Rightarrow$ (b):
If $M^{**}$ is free, then
so is $M^{***} =
(M^*)^{**}$. So $M^*$
is free by Lemma~\ref{lem:mstar}.\\
(a) $\Rightarrow$ (c):
Suppose that $M \in \FF(R)$.
By Proposition~\ref{prop:closureft}(a)
and Lemma~\ref{lem:isocodim1},
we may replace $M$ by $M^{**}$
and assume that $M$ is
reflexive. So we need to show
that $M$ is free. Since
$R$ is regular, the ring
extension $R \to \upe{R}$
is flat. By Lemma~\ref{lem:basechangerfl},
we have $M(e)=(M \otimes_R \upe{R})^{**}
= M \otimes_R \upe{R}$ for
all $e \geqslant 0$.
By Corollary~\ref{cor:periodicity},
we have $M(e') \cong M(e'+e)$ for some
$e' \geqslant 0$ and $e>0$.
First, suppose that $e'=0$,
so that $M \cong M(e)$.
Consider a minimal
free resolution
\[
  \cdots \to
  F_1 \xrightarrow{A}
  F_0 \to M \to 0
\]
of $M$ given by the matrix $A$.
Let $I(M) \subseteq \mathfrak{m}$
be the Fitting
ideal of $M$ generated by
the entries of $A$ and let
$q=p^e$. Tensoring with $\upe{R}$
gives a free resolution
\[
  \cdots \to
  F_1 \xrightarrow{A^{[q]}}
  F_0 \to M(e) \to 0,
  \tag{\st}
\]
so $I(M(e)) = I(M)^{[q]}$. Since
$M \cong M(e)$, we have
$I(M) = I(M(e)) =
I(M)^{[q]}$, so
$I(M) = I(M)^{q}$.
By Nakayama's Lemma, we
have $I(M)=0$. So $A=0$,
and $M^{**}=M$ is free.\\
In general, the above shows
that $M(e)$ is free for some
$e \geqslant 0$ by
Corollary~\ref{cor:meplusf}.
Suppose that $e>0$. Since
(\st) is a free resolution
of the free module $M(e)$,
it is a direct sum of a
trivial complex and the
resolution $0 \to F \to
M(e) \to 0$. So the entries
in $A^{[q]}$ are either 1 or 0.
Since $A$ has entries in
$\mathfrak{m}$ and $R$ is a
domain, we must have $A=0$,
so again $M$ is free.\\
(c) $\Rightarrow$ (a):
Suppose that $M$ is free in
codimension 1.
If $M^{**}$ is free, then
$M^{**} \in \FF(R)$, so
$M \in \FF(R)$ by
Proposition~\ref{prop:closureft}(a).
\end{proof}

\begin{thm} \label{thm:addcm}
Let $\phi \colon R
\to S$
be a finite homomorphism
of normal domains such that
$\phi$ is \'{e}tale in
codimension $1$ and splits as a map of
$R$-modules. Let $N,L$ be $R$-modules such that
$N$ is $(S_2)$ and that locally, $(N,L)$ is an
abundant pair and $N \in \add(L)$.
Then $L \in \add_R(\CM(S))$.
\end{thm}

\begin{proof}
Since $S$ is of finite $F$-type over $R$,
$\chom_R(S,L)$ is C-M by Theorem~\ref{thm:mainstrong},
so $\chom_R(S,L)$ $\in \CM(S)$. Since $\phi$
splits, $L$ is a direct summand of $\chom_R(S,L)$.
\end{proof}

\begin{cor}\label{quotient}
Let $\phi \colon R
\to S$
be a finite homomorphism
of normal domains such that
$\phi$ is \'{e}tale in
codimension $1$ and splits as a map of
$R$-modules and that $S$
is regular (for example, if $R$
is a quotient singularity as in example \ref{HNex}). If
$M \in \FF(R)$, then $M^* \in \add_R(S)$. If $S^* = \Hom_R(S,R)\cong S$ (again, for example if $R$
is a quotient singularity as in \ref{HNex}) then $\FF(R)= \add_R(S)$. 
\end{cor}

\begin{proof}
Let $N = M \otimes_R S$. If
$M \in \FF(R)$, then $N \in \FF(S)$
by Lemma~\ref{lem:basechangeft}.
Then $N^*$ is free by
Lemma~\ref{lem:regft}. We have
$N^* = \chom_S(M \otimes_R S,
S) = \chom_R(M,\chom_S(S,S))
= \chom_R(M,S)$. Since $f$ splits,
$\chom_R(M,R)$ is a direct
summand of $\chom_R(M,S)$, so
$M^* \in \add_R(S)$. If $S^*\cong S$, then it follows that $\FF(R) \subseteq \add_R(S)$, and $S\in \FF(R)$ by Example \ref{eg:modft}.
\end{proof}

\begin{cor}\label{cor:inva}
Suppose that $k$ is
algebraically closed.
Let $S = k[[x_1,\dots,
x_d]]$. Let $G$ be a finite
subgroup of $GL(d,k)$ that
contains no pseudo-reflections
such that the order of $G$
is coprime to $p$. Let
$R = S^G$. Then $\FF(R) = \add_R S$. 
\end{cor}

\begin{proof}
Use Corollary \ref{quotient}.
\end{proof}

\begin{thm} \label{thm:cift}
Suppose that $R$ is a
complete intersection
and $M$ is an $R$-module that
is free in codimension 2.
Then $M \in \FF(R)$ if and only if
$M^{**}$ is free.
\end{thm}

\begin{proof}
The proof of ``if'' is given
by (c) $\Rightarrow$ (a) in
Lemma~\ref{lem:regft}. For
the ``only if'' part, as
in the proof of Lemma~\ref{lem:regft}
(a) $\Rightarrow$ (c), it
suffices to assume that $M$
is reflexive and show that $M$
is free. The proof is by
induction on $d$. For $d \geqslant 3$,
$M$ is free on $\spec R \setminus
\{ \m \}$ by induction. Let
$r = \dep M \geqslant 2$.
Suppose that $r < d$. Then
by \cite[Proposition~4.14]{DS},
we have $\CH^r_{\m} (F^e_R(M))
\cong \CH^r_{\m} (F^e_R(M)^{**})
= \upe{\CH^r_{\m}} (M(e))$ as
in Theorem~\ref{thm:mainweak}, so
\[
  \lim_{e \to \infty}
  \frac{\lambda_R(\CH^r_{\m}(F^e_R(M))}
  {p^{e(d+\alpha(R))}}
  = \lim_{e \to \infty}
  \frac{\lambda_R(\upe{\CH^r_{\m}} (M(e)))}
  {p^{e(d+\alpha(R))}}
  = \lim_{e \to \infty}
  \frac{p^{e \alpha(R)}
  \lambda_R(\CH^r_{\m} (M(e)))}
  {p^{e(d+\alpha(R))}}
  = 0
\]
by the assumption that $M \in \FF(R)$.
Then \cite[Theorem~4.12]{DS}
gives $\pd_R M < d-r$. By the
Auslander-Buchsbaum formula,
we have $\dep M > r$, a
contradiction. So $r=d$, and again
\cite[Theorem~4.12]{DS} with
$k = r-1$ shows that $\pd_R M=0$,
so $M$ is free.
\end{proof}

\begin{eg}
Let $R = k[[x,y,z]]/(xy-z^2)$
and $M$ be the ideal $(x,z)$.
Then $[M]$ has order 2 in
$\Cl(R)$ and so $M \in \FF(R)$
by Example~\ref{eg:modft},
but $M^{**}$ is not free, so
the condition that $M$ is
locally free in codimension 2
is necessary in
Theorem~\ref{thm:cift}.
\end{eg}

\section{\texorpdfstring{$F$}{F}-abundant
pairs and modules} \label{sec:abun}
In this section we give many examples of $F$-abundant pairs and modules.  
\begin{eg} \label{eg:abdtpair}
(a) If $\sdim R \geqslant 1$,
in particular if $R$
is strongly $F$-regular of
dimension $\geqslant 1$,
then $(R,R)$ is an
abundant pair.\\
(b) \cite[Proposition~2.3]{Kunz}
shows that $\alpha(R_P) = \alpha(R)
+ \dim (R/P)$.
Let $N,L$ be as in
Definition~\ref{defn:intro}
and $P \in \spec R$.
If $\liminf_{e \to \infty}
p^{e(\alpha(R) +
\dim(R/P))}/b_e
= 0$, then $(N_P,L_P)$ is
an abundant pair.\\
(c) $F$-contributors for
modules of finite $F$-representation
type, as in \cite[Section~2
]{Yao}.
\end{eg}

\begin{eg} (\cite[Example 6.1]{DS})
Let $k$ be an algebraically closed field of
characteristic $p>2$. Consider the hypersurface
$R=k[[x,y,u,v]]/(xy-uv)$. Then
every MCM $R$-module is $F$-abundant.
\end{eg}

\begin{eg}\label{HNex}
Let $k$ be an algebraically closed field of characteristic $p>0$ and $V$ be a $k$-vector space of dimension $d$. Let $S$ be the symmetric algebra of $V$. Let $G$ be a finite subgroup of $GL(V)$ without pseudo-reflections such that $|G|$ is coprime to $p$. Let $R=S^G$ be the ring of invariants. Let $V_0,\cdots, V_n$ be a complete set of irreducible representations of $G$ over $k$. Let $M_i= (S\otimes_k V_i)^G$. It is classical that $\add_R(S)=\{M_0,M_1,\dots,M_n\}$.  Also, by the main results of \cite{HN}:
\begin{enumerate}[align=left,leftmargin=*]
\item $S$, and hence all of $M_0, M_1,\dots, M_n$, are modules of finite $F$-type. Note that $\rk M_i =\dim_k V_i$, so we have many examples of modules of finite $F$-type which are not ideals. 
\item $(M_i,M_j)$ is an $F$-abundant pair for all $0\leq i,j\leq n$. 
\end{enumerate} 
\end{eg}

\begin{cor} \label{cor:thmsharp}
Let $R,N$ be as in
Theorem~\ref{thm:mainweak}.
Suppose that $\liminf_{e \to \infty}
p^{e(\alpha(R)+d-3)}/b_e=0$
as in Example~\ref{eg:abdtpair}
with $N=L$
(for example, when $d=3$ and
$(N,N)$ is an abundant pair).
Then $N$ is MCM. In particular,
if $R$ is regular, then $N$
is free.
\end{cor}

\begin{proof}
Theorem~\ref{thm:mainstrong}
with $M=R$ and $e=0$ shows that
$N$ is MCM.
\end{proof}

\begin{eg} \label{eg:sharp_kxyz}
Let $k$ be a perfect field,
$R=k[[x,y,z]]$ and $M$ be
the ideal $(x,y)$. Then
$(M,M)$ is an abundant pair,
but $M$ is not $(S_2)$.
So the assumption in
Corollary~\ref{cor:thmsharp}
for $N$ to be $(S_2)$ cannot
be weakened.
\end{eg}

\begin{proof}
Let $C=k[[z]]$. Consider the
exact sequence
\[
  0 \to M \to R \to C \to 0
\]
Let $q=p^e$.
Since $\upe{-}$ is an exact
functor, we have
\[
  \xymatrix{
  0 \ar[r]
  & \upe{M} \ar[r]
  & \upe{R} \ar[r] \ar@{=}[d]
  & \upe{C} \ar[r] \ar@{=}[d]
  & 0\\
  && R^{\oplus q^3} \ar[r]
  & C^{\oplus q} \ar[r]
  & 0
  }
\]
It follows that $\upe{M}$ has
exactly $q$ copies of $M$,
so $(M,M)$ is an abundant pair.
Since $1 = \dep M_M < \dim M_M
= 2$, $M$ is not $(S_2)$.
\end{proof}

\begin{eg}
Let $k$ be a perfect field,
$R=k[[x_1,x_2,\dots,x_{d-3},u,v,w]]$,
$C=R/(u,v,w)$ and $M=\Omega^2 C$,
the second syzygy of $C$. Then
$M$ is $(S_2)$ and $(M,M)$ is
an abundant pair.
So the assumptions in
Corollary~\ref{cor:thmsharp}
cannot be weakened.
\end{eg}

\begin{proof}
Let $q=p^e$. Then as in
Example~\ref{eg:sharp_kxyz},
we have
\[
  \xymatrix{
  0 \ar[r]
  & \upe{M} \ar[r]
  & \upe{(R^3)} \ar[r] \ar@{=}[d]
  & \upe{R} \ar[r] \ar@{=}[d]
  & \upe{C} \ar[r] \ar@{=}[d]
  & 0\\
  && R^{\oplus 3q^d} \ar[r]
  & R^{\oplus q^d} \ar[r]
  & C^{\oplus q^{d-3}} \ar[r]
  & 0
  }
\]
It follows that $\upe{M}$ has
exactly $q^{d-3}$ copies of $M$.
Then as in Corollary~\ref{cor:thmsharp},
we have $\liminf_{e \to \infty}
q^{(\alpha(R)+d-3)}/b_e
= q^{d-3}/q^{d-3} = 1$.
Since $M$ is a second syzygy,
it is $(S_2)$. Since $\dep C
= d-3$, we have $\dep M = d-1
\neq d$, so $M$ is not free.
\end{proof}

\section{Geometric applications}
\label{sec:vecbun}

\setcounter{thm}{-1}

\begin{discuss} \label{dis:vecbun}
Let $(A,\m)$ be a standard
graded ring that is $(S_2)$
of dimension at least 2. We
first fix some notation and
record some results from
\cite{Hor,Mz}. Let $R = A_{\m}$.
Let $X = \bproj(A)$, $Y =
\bspec(R) \setminus \{\m_{\m}\}$
and $Z = \bspec(R)$.
Let $\iota \colon Y \to Z$
be the inclusion morphism.
There is also an affine
surjective morphism
$\pi \colon Y \to X$
(\cite[Proposition~I.5]{Mz}).
Corollary I.6 of \cite{Mz}
states that for every
$A$-graded module $M$ we have
sheaf cohomology isomorphisms
\[
  \oplus_{d \in \mathbb{Z}}
  H^i (X,\tilde{M}(d))
  \xrightarrow{\simeq}
  H^i(Y,\tilde{M_{\m}}|_Y)
\]
for $i \geqslant 0$. For $1
\leqslant i \leqslant d-1$,
we have $H^i(Y,\tilde{M_{\m}}|_Y)
\cong H^{i+1}_{\m} (M_{\m})$.

Let $V(\cdot)$ denote the
category of vector bundles
over a scheme. Let $\Gamma$
be the global section functor.
If $\scr{G} \in V(Y)$,
then $\iota_* \scr{G}$ is
a coherent $\scr{O}_Z$-module.
Let $\psi = \Gamma \circ
\pi^*$. Then we have maps
\[
  V(X) \xrightarrow{\pi^*}
  V(Y) \xrightarrow{\Gamma}
  \{ M \in \refl(R) \mid
  M \text{ is locally free on }
  Y \}
\]
By \cite[Theorem~1.3]{Hor},
$\dep \Gamma(\scr{G}) \geqslant 2$.
$\Gamma$ satisfies the
property $\Gamma(\scr{G}_1 \otimes
\scr{G}_2) = (\Gamma(\scr{G}_1)
\otimes_R \Gamma(\scr{G}_2))^{**}$.
Let $\scr{F}_i \in V(X)$, $i=1,2$,
be indecomposable such that
$\psi(\scr{F}_i)$ are isomorphic
up to free summand. Then by
\cite[Proposition~9.5]{Hor},
$\scr{F}_1 \cong \scr{F}_2 (m)
= \scr{F}_2 \otimes \scr{O}_X(m)$
for some $m$.
\end{discuss}

\begin{defn}
(\cite[Introduction]{BK})
Let $X$ be a smooth projective
variety defined over a field $k$
of characteristic $p>0$. Let
$\phi \colon X \to X$ be the
absolute Frobenius morphism.
Then a vector bundle $\scr{F}
\in V(X)$ is $(e,f)$-Frobenius
periodic, or $(e,f)$-F periodic
in short, if there are $e<f$
such that $(\phi^e)^*(\scr{F})
\cong (\phi^f)^*(\scr{F})$.
\end{defn}

\begin{lemma} \label{lem:fperiodic}
Suppose that $\scr{F} \in V(X)$ is
$(e,f)$-F periodic. Let $X,Y$
be as in \ref{dis:vecbun}.
Let $M= \psi(\scr{F})$.
Then $\scr{F}$ is $(e_1,e_2)$-F
periodic iff $M(e_1) \cong M(e_2)$.
\end{lemma}

\begin{proof}
For ``only if'', we have
$\psi \big( (\phi^e)^*(\scr{F}) \big)
= \psi (\scr{F} \otimes (\phi^e)^*(X))
= \big( \psi (\scr{F}) \otimes
(\phi^e)^*(X) \big)^{**}
= M(e)$ as in \ref{dis:vecbun}.
The ``if'' part comes from
sheafification.
\end{proof}

Next, we discuss a generalization of one of the results in \cite[Theorem 3.1]{PS}. For that we need to recall some notation. Let $R$ be a $F$-finite normal domain and let $D$ be a $\Q$-divisor. 

In the next result we are able to remove the condition that the characteristic $p$ is coprime to $r$ as in  \cite[Theorem 3.1]{PS}. 
We follow the same trick  as in \cite{PS}, with a crucial difference suggested by our approach in this paper: the reflexive module representing a torsion element in the class group has finite $F$-type. 

\begin{thm}\label{PS}
Let $R$ be an $F$-finite normal domain with perfect residue field and $X =\Spec R$. Let $\Delta$ be a $\Q$-divisor on $X$ such that the pair $(X, \Delta)$ is strongly $F$-regular. Let $D$ be an integral divisor such that $rD \sim r\Delta'$ for some integer $r>0$ and $0\leq \Delta'\leq\Delta$. Then $\mathscr{O}_X(-D)$ is Cohen-Macaulay.
\end{thm}

\begin{proof}
Since we also have that the pair $(X,\Delta')$ is strongly $F$-regular, one may assume $\Delta'=\Delta$. Now, this assumption implies that there is a decomposition of $R$-modules ($q=p^e$) (\cite[Lemma 3.5]{BST}):
$$F_*^e\mathscr{O}_X((q-1)\Delta) = \mathscr{O}_X^{n_e} \oplus N_e$$ 
such that $\liminf_{e \to \infty} \frac{n_e}{q^d} >0$.  
Twisting by $\mathscr{O}_X(-D)$, reflexifying, we get a decomposition:
$$F_*^e\mathscr{O}_X((q-1)(\Delta-D)-D) = \mathscr{O}_X(-D)^{n_e} \oplus N'_e$$ 
The key point now is that as $r(\Delta-D)\sim 0$, there are only finitely many isomorphism classes of the modules $\mathscr{O}_X((q-1)(\Delta-D)-D)$. Let $M$ be the direct sum of all these modules and $I=\mathscr{O}_X(-D)$, what we have is precisely:
$$^eM \cong  I^{n_e}\oplus P_e $$
with $\liminf_{e \to \infty} \frac{n_e}{q^d} >0$. Lemma \ref{lem:Yao} now forces $I$ to be Cohen-Macaulay.
\end{proof}

\section{Acknowledgments} 
The first author would like to thank Kevin Tucker and Karl Schwede for patiently explaining to him many basic facts on the subjects this note.

\end{document}